\documentclass{article}
\usepackage{tikz}
\usepackage{listings}
\usepackage{color}
\usepackage{graphicx}            
\usepackage{amsthm,amsfonts,amsmath,amssymb}
\usepackage{array,blkarray,multirow, enumerate}
\usepackage{xcolor}

\newtheorem{definition}{Definition}
\newtheorem{theorem}{Theorem}

\newtheorem{remark}[theorem]{Remark}
\newtheorem{claim}[theorem]{Claim}

\newcounter{mycount}

\def\0{\mathbf{0}}

\usepackage[normalem]{ulem}

\begin{document}
\lstset{language=Python}          

\title{Triharmonic Curves in $f$-Kenmotsu Manifolds}
\author{Şerife Nur Bozdağ}
\maketitle
\begin{abstract}
    The aim of this paper is to study triharmonic curves in three dimensional $f$-Kenmotsu manifolds. We investigate necessary and sufficient conditions for Frenet curves, and specifically for slant and Legendre curves to be triharmonic. Then we prove that triharmonic Frenet curves with constant curvature are Frenet helices in three dimensional $f$-Kenmotsu manifolds. Next, we give a nonexistence theorem that there is no triharmonic Legendre curve in three dimensional $f$-Kenmotsu manifolds.

\end{abstract}

\bigskip


\bigskip


\section{Introduction}
In 1964, Eells and Sampson \cite{EellsSamp2} introduced $k$-harmonic maps (polyharmonic maps) as the critical points of the $k$-energy functional defined by 
    \begin{equation} \label{ef}
    E_{k}^{ES}(\psi)=\dfrac{1}{2}\int_{M} \mid (d+d^{*})^{r}\psi\mid ^{2}v_{g},\ \ r \geq 1
    \end{equation}
    for $\psi \in C^{\infty}(M,\bar{M})$ where $M,\bar{M}$ are Riemannian manifolds. Here, when $\gamma:I \subset \mathbb{R}\rightarrow M $ is an arc-length parametrized curve, substutiting $ \gamma^{'}=T $ where $T$ is the unit tangent vector field of $\gamma$, the Euler-Lagrange equation reduces to
    \begin{equation} \label{tauk}
    \tau_{k}(\gamma)=\nabla_{T}^{2k-1}T+\sum_{l=0}^{k-2} (-1)^{l}R^{M}(\nabla_{T}^{2k-3-l}T,\nabla_{T}^{l}T)T=0, \hspace{0.7cm} k \geq 1. 
    \end{equation}
    
\indent Solutions of equation (\ref{tauk}) are called as $k$-harmonic curves. Particularly, for $k \geq 1$, any harmonic curve is a polyharmonic curve. Then $k$-harmonic curves called as proper $k$-harmonic if they are not harmonic, \cite{Montaldo}.\\
\indent Besides, Eells and Lemaire, studied in detail $k$-energy functionals and $k$-harmonic maps, in \cite{EllsLamaire} and Maeta investigated $k$-harmonic curves into Riemannian manifolds with constant sectional curvature, in \cite{maeta}.\\
\indent The theory of harmonic maps has been applied to a wide variety of fields, primarily to differential and Riemannian geometry. Harmonic maps between two Riemannian manifolds are critical points of the energy functional $$E(\psi)=\frac{1}{2}\int_{M}\mid d\psi\mid^{2}v_{g}.$$
Namely, $\psi:(M,g)\rightarrow (\bar{M},\bar{g})$ is called as harmonic map if $$\tau (\psi )=-d^{*}d\psi=trace\nabla d\psi =0$$ where $M,\bar{M}$ are Riemannian manifolds. Here $\tau (\psi ),$ which is the tension field of $\psi,$ is the Euler-Lagrange equation of the energy functional $E(\psi)$, $d$ is the exterior differentiation, $d^{*}$ is the codifferentiation, $\nabla$ is the connection induced from the Levi-Civita connection $\nabla^{\bar{M}}$ of $\bar{M}$ and the pull-back connection $\nabla^{\psi},$ \cite{EellsSamp,maeta}.\\
\indent  On the other hand, Jiang studied first and second variation formulas of the bienergy functional $E_{2}(\psi)$ whose critical points are called as biharmonic maps, \cite{Jiang}. There have been a rich literature on biharmonic maps like as harmonic maps.\\
\indent  In this study, we focus on arc-length parametrized triharmonic Frenet curves, which are the solutions of (\ref{tauk}) for $k=3$, given as the following equality
    \begin{equation} \label{trico}
    \tau_{3}(\gamma)=\nabla_{T}^{5}T+R^{M}(\nabla_{T}^{3}T,T)T-R^{M}(\nabla_{T}^{2}T,\nabla _{T}T)T=0.
    \end{equation}
    It should be noted that every harmonic curve is also a triharmonic curve. Nevertheless, biharmonic curves are not necessaryly triharmonic curves and vice versa triharmonic curves not need to be biharmonic.
    Thus, as Montaldo and Pampano pointed out in \cite{Montaldo}, that the study of triharmonic curves could be a completely different problem than biharmonic curves in general.
    \\ \indent Although specific studies on harmonic and biharmonic curves are quite rich, triharmonic curves are generally briefly examined in studies under the title of $k$-harmonic curves. However, it has become more interesting exclusively with this recent work by Montaldo and Pampano in \cite{Montaldo} which they investigated the possibility of constructing triharmonic curves in a Riemannian manifold with nonconstant curvature. For these reasons, the purpose of this paper is to study triharmonic curves with constant and nonconstant curvature in $f$-Kenmotsu manifold in order to add a new perspective to study of triharmonic curves and to motivate readers work on this new topic.\\\indent 
    Our study consists of the following sections. Section 2 is reserved for an introduction to $f$-Kenmotsu manifolds, slant curves and Legendre curves. In Section 3, triharmonicity conditions of a Frenet curve are obtained and these conditions are studied under various cases. Then, detailed interpretations are given about the triharmonicity conditions of slant and Legendre curves.

\section{$f$-Kenmotsu manifolds} 
A differentiable manifold $M^{2n+1}$ is called almost contact metric manifold with the almost contact metric structure $(\varphi,\xi,\eta,g),$  if it admits a tensor field $\varphi$ of type $(1,1)$, vector field $\xi$ , $1$-form $\eta$ and Riemannian metric tensor field $ g $ satisfying the following conditions; 
\begin{eqnarray}
\varphi ^{2}=-I+\eta \otimes \xi, \quad\quad \quad \quad \quad \nonumber \\ \eta (\xi )=1, \quad \varphi \xi =0, \quad \eta\circ\varphi=0, \quad \eta(X)=g(X,\xi)   \\ g(\varphi X,\varphi Y)=g(X,Y)-\eta(X)\eta(Y) \quad \nonumber 
  \label{1}
\end{eqnarray} 
where $ I $ denotes the identity transformation and $  X,Y \in \Gamma(TM)$, \cite{blair}.\\ An almost contact metric manifold is said be $ f$-Kenmotsu manifold if the Levi-Civita connection $ \nabla $  of $ g $ satisfies
\begin{eqnarray}
\left( \nabla _{X}\,\varphi \right) Y&=&f \left( g(\varphi X,Y\right) \xi-\eta (Y)\varphi X),\\
 \nabla _{X}\,\xi&=&f(X-\eta (X)\xi)
\end{eqnarray}
where $ f \in C^{\infty}(M)$ such that $df \wedge \eta=0$ and $X,Y \in \Gamma(TM)$,\cite{kenmo}. \\
It is well known that on a $3$-dimensional Riemannian manifold, curvature tensor field $R$ given as below; 
\begin{eqnarray}
R(X,Y)Z&=&g(Y,Z)QX-g(X,Z)QY+Ric(Y,Z)X \nonumber\\  &-&Ric(X,Z)Y-\frac{r}{2}\left\lbrace g(Y,Z)X-g(X,Z)Y\right\rbrace \nonumber
\end{eqnarray}
where  $X,Y,Z\in \Gamma(TM),$ $Ric$ is the Ricci tensor, $Q$ is the Ricci operator and $r$ is the scalar curvature of $M$, \cite{mangione}.\\
Thus for a $3$-dimensional $ f $-Kenmotsu manifold, the curvature tensor field becomes;
\begin{small}
\begin{eqnarray} \label{curvature} 
R(X,Y)Z&=&( \frac{r}{2}+2( f^{2}+f^{'}))
\big( g(Y,Z)X-g(X,Z)Y \big)\nonumber \\  &-&(\frac{r}{2}+3(f^{2}+f^{'}))\big(g(Y,Z)\eta(X)\xi-g(X,Z)\eta(Y)\xi\-\eta(X)\eta(Z)Y+\eta(Y)\eta(Z)X\big)\nonumber \\  
\end{eqnarray}
\end{small} 
Now let us consider an arc-length parametrized curve $\gamma:I \subset \mathbb{R}\rightarrow M$ in a $n$-dimensional Riemannian manifold $(M,g).$ If $E_{1}, E_{2},...,E_{r}$ orthonormal vector fields satisfying
 \begin{eqnarray}
 E_{1}&=&\gamma^{'}\nonumber\\ \nonumber
 \nabla _{T}E_{1},&=&k_{1}E_{2},   \\ \nonumber
 \nabla _{T}E_{2}&=&-k_{1}E_{1}+k_{2}E_{3},   \\ \nonumber
 &...&\\ \nonumber
 \nabla_{T}E_{r}&=&-k_{r-1}E_{r-1}  \nonumber 
 \end{eqnarray}
 along $\gamma,$ then $\gamma$ is called a Frenet curve of osculating order $r$. Here $k_{1},...,k_{r-1}$ are positive functions on $I$ and $1\leq r \leq n.$ With this for $k_{1},...,k_{r-1}$ are non-zero positive constants; a Frenet curve of osculating order $1$ is called geodesic, of osculating order $2$ is called circle if $k_{1}=constant>0,$ and of osculating order $r\geq 3$ is called helix of order $r$, \cite{Montaldo}.\\
\indent Now let remind the definitions and properties of slant and Legendre curves.
\begin{definition}
$\gamma(s):I \subset \mathbb{R}\rightarrow M$ is called a slant curve if the contact angle $ \theta:I \rightarrow [0, 2\pi)$ of 
 given by $cos \theta(s) = g(T(s), \xi) $
is a constant function.\\ In particular, if $\theta=\dfrac{\pi}{2} (or \dfrac{3\pi}{2})$  then $\gamma(s)$ is called a Legendre curve, \cite{munteanu}.
\end{definition}

\begin{remark}
For a slant curve in $f$-Kenmotsu manifolds, we have \cite{munteanu}
\begin{eqnarray}
\eta(N)=-\dfrac{f}{k_{1}}(sin\theta)^{2}, \nonumber
\end{eqnarray}
where $\mid sin\theta\mid\leq min{\dfrac{k_{1}}{f}}$ and
\begin{eqnarray}
\eta(B)=\dfrac{\mid sin\theta \mid}{k_{1}}\sqrt{k_{1}^{2}-f^{2}(sin\theta)^{2}}.\nonumber
\end{eqnarray}

\end{remark}
\begin{remark} \label{legremark}
For a Legendre curve in $f$-Kenmotsu manifolds, we have \cite{munteanu, Perk}
\begin{eqnarray}
N=-\xi,\ \ k_{1}=f \mid _{\gamma},\ \  k_{2}=0, \nonumber
\end{eqnarray}
and 
\begin{eqnarray}
\eta(T)=0,\ \ \eta(N)=-1,\ \  \eta(B)=0. \nonumber
\end{eqnarray}
\end{remark}

\section{Triharmonic curves in $f$-Kenmotsu manifolds}

Let $(M,\varphi,\xi,\eta,g)$ be a $3$-dimensional $f$-Kenmotsu manifold and $\gamma(s):I \subset \mathbb{R}\rightarrow M$ be a non-geodesic Frenet curve parametrized by arclength $s$. The Serret-Frenet frame $\left\lbrace T=\gamma^{'}(s), N, B\right\rbrace $ along $\gamma$, respectively; the tangent, the principal normal and the binormal vector fields. Then the Serret-Frenet formulas are given as;
\begin{equation}
\begin{cases}
\nabla _{T}T=k_{1}N  \vspace{0.2cm} \\
\nabla _{T}N=-k_{1}T+k_{2}B \vspace{0.2cm} \label{deri}\\
\nabla _{T}B=-k_{2}N 
\end{cases}
\end{equation}
where $k_{1}$ and $k_{2}$ are the curvature and the torsion of the curve, respectively.\\ By using Serret-Frenet formulas given in (\ref{deri}) and applying $\nabla _{T}T=k_{1}N$ as many times as needed, we obtained 
\begin{eqnarray}
\nabla _{T}^{2}T&=&-k_{1}^{2}T+k_{1}^{'}N+k_{1}k_{2}B,  \label{nabla2TT} \\ 
\nabla_{T}^{3}T&=&-3k_{1}k_{1}^{'}T+(-k_{1}^{3}-k_{1}k_{2}^{2}+k_{1}^{''})N+(2k_{1}^{'}k_{2}+k_{1}k_{2}^{'})B,  \label{nabla3TT} \\ 
\nabla_{T}^{4}T&=&(k_{1}^{4}+k_{1}^{2}k_{2}^{2}-4k_{1}^{''}k_{1}-3(k_{1}^{'})^{2})T\nonumber\\&-&(6k_{1}^{2}k_{1}^{'}+3k_{2}^{2}k_{1}^{'}+3k_{1}k_{2}k_{2}^{'}-k_{1}^{'''})N\nonumber\\&+& (3k_{1}^{''}k_{2}+k_{1}k_{2}^{''}+3k_{1}^{'}k_{2}^{'}-k_{2}k_{1}^{3}-k_{1}k_{2}^{3})B     
\end{eqnarray}
and finally
\begin{eqnarray}
\nabla_{T}^{5}T&=&(10k_{1}^{3}k_{1}^{'}+5k_{1}k_{1}^{'}k_{2}^{2}-5k_{1}k_{1}^{'''}-10k_{1}^{'}k_{1}^{''}+5k_{2}k_{2}^{'}k_{1}^{2})T  \\ 
&+&\big(k_{1}^{5}+2k_{1}^{3}k_{2}^{2}-10k_{1}^{2}k_{1}^{''}-15k_{1}(k_{1}^{'})^{2}-12k_{2}k_{2}^{'}k_{1}^{'}\nonumber\\&&-6k_{2}^{2}k_{1}^{''}-3k_{1}(k_{2}^{'})^{2}-4k_{1}k_{2}k_{2}^{''}+k_{1}k_{2}^{4}+k_{1}^{(4)}\big) N\nonumber\\
&+&(-9k_{1}^{2}k_{2}k_{1}^{'}-4k_{2}^{3}k_{1}^{'}-6k_{1}k_{2}^{2}k_{2}^{'}+4k_{1}^{'''}k_{2}+6k_{1}^{''}k_{2}^{'}\nonumber\\&&+4k_{1}^{'}k_{2}^{''}+k_{1}k_{2}^{'''}-k_{2}^{'}k_{1}^{3})B. \nonumber
\end{eqnarray} 
On the other hand, by substutiting (\ref{nabla2TT}), (\ref{nabla3TT}) and $\nabla _{T}T=k_{1}N$ into the curvature tensor formula (\ref{curvature}),  we get

\begin{eqnarray}
R^{M}(\nabla_{T}^{3}T,T)T&=&\bigg((k_{1}^{''}-k_{1}^{3}-k_{1}k_{2}^{2})\big(\dfrac{r}{2}+3(f^{2}+f^{'}))\eta(N)\eta(T)\nonumber\\&&+ (2k_{1}^{'}k_{2}+k_{1}k_{2}^{'})(\dfrac{r}{2}+3(f^{2}+f^{'}))\eta(B)\eta(T)\bigg)T\nonumber \\&+&
\bigg((k_{1}^{''}-k_{1}^{3}-k_{1}k_{2}^{2})\big(\dfrac{r}{2}+2(f^{2}+f^{'}))\nonumber\\&&-(k_{1}^{''}-k_{1}^{3}-k_{1}k_{2}^{2})(\dfrac{r}{2}+3(f^{2}+f^{'}))\eta(T)^{2}\bigg)N
\nonumber \\&+&
\bigg((2k_{1}^{'}k_{2}+k_{1}k_{2}^{'})\big(\dfrac{r}{2}+2(f^{2}+f^{'}))\nonumber\\&&-(2k_{1}^{'}k_{2}+k_{1}k_{2}^{'})(\dfrac{r}{2}+3(f^{2}+f^{'}))\eta(T)^{2}\bigg)B\nonumber \\&-&\bigg((k_{1}^{''}-k_{1}^{3}-k_{1}k_{2}^{2})\big(\dfrac{r}{2}+3(f^{2}+f^{'}))\eta(N)\nonumber\\&&+ (2k_{1}^{'}k_{2}+k_{1}k_{2}^{'})(\dfrac{r}{2}+3(f^{2}+f^{'}))\eta(B)\bigg)\xi
\end{eqnarray}

and

\begin{eqnarray}
R^{M}(\nabla_{T}^{2}T,\nabla_{T}T)T&=&\bigg(k_{1}^{3}\big(\dfrac{r}{2}+3(f^{2}+f^{'}))\eta(N)\eta(T)\bigg)T\nonumber \\&+&
\bigg(k_{1}^{3}\big(\dfrac{r}{2}+2(f^{2}+f^{'}))-k_{1}^{3}(\dfrac{r}{2}+3(f^{2}+f^{'}))\eta(T)^{2}\nonumber\\&&+k_{1}^{2}k_{2}(\dfrac{r}{2}+3(f^{2}+f^{'})\eta(B)\eta(T)\bigg)N
\nonumber \\&-&
\bigg((k_{1}^{2}k_{2})\big(\dfrac{r}{2}+3(f^{2}+f^{'}))\eta(N)\eta(T)\bigg)B\nonumber \\&-&\bigg(k_{1}^{3}\big(\dfrac{r}{2}+3(f^{2}+f^{'}))\eta(N)\bigg)\xi.
\end{eqnarray}

After these calculations, by substutiting these equations to the following formula  $$\tau_{3}(\gamma)=\nabla_{T}^{5}T+R^{M}(\nabla_{T}^{3}T,T)T-R^{M}(\nabla_{T}^{2}T,\nabla _{T}T)T=0,$$ we obtained the triharmonicity condition of a Frenet curve in a 3-dimensional $f$-Kenmotsu manifold as follows:

 \begin{eqnarray} 
\tau_{3}(\gamma)&=&\bigg[10k_{1}^{3}k_{1}^{'}+5k_{1}k_{1}^{'}k_{2}^{2}-5k_{1}k_{1}^{'''}-10k_{1}^{'}k_{1}^{''}+5k_{2}k_{2}^{'}k_{1}^{2}\nonumber\\&&+ (k_{1}^{''}-2k_{1}^{3}-k_{1}k_{2}^{2})\big(\dfrac{r}{2}+3(f^{2}+f^{'}))\eta(N)\eta(T)\nonumber\\&& + (2k_{1}^{'}k_{2}+k_{1}k_{2}^{'})(\dfrac{r}{2}+3(f^{2}+f^{'}))\eta(B)\eta(T) \bigg]T\nonumber \\&+&
\bigg[ k_{1}^{5}+2k_{1}^{3}k_{2}^{2}-10k_{1}^{2}k_{1}^{''}-15k_{1}(k_{1}^{'})^{2}-12k_{2}k_{2}^{'}k_{1}^{'}-6k_{2}^{2}k_{1}^{''}\nonumber\\&&-3k_{1}(k_{2}^{'})^{2}-4k_{1}k_{2}k_{2}^{''}+k_{1}k_{2}^{4}+k_{1}^{(4)}\nonumber\\&&+(k_{1}^{''}-2k_{1}^{3}-k_{1}k_{2}^{2})\big((\dfrac{r}{2}+2(f^{2}+f^{'}))-(\dfrac{r}{2}+3(f^{2}+f^{'}))\eta(T)^{2}\big)\nonumber\\&&-k_{1}^{2}k_{2}(\dfrac{r}{2}+3(f^{2}+f^{'})\eta(B)\eta(T) \bigg]N
\nonumber \\&+&
\bigg[ -9k_{1}^{2}k_{2}k_{1}^{'}-4k_{2}^{3}k_{1}^{'}-6k_{1}k_{2}^{2}k_{2}^{'}+4k_{1}^{'''}k_{2}+6k_{1}^{''}k_{2}^{'}+4k_{1}^{'}k_{2}^{''}+k_{1}k_{2}^{'''}\nonumber\\&&-k_{2}^{'}k_{1}^{3}+(2k_{1}^{'}k_{2}+k_{1}k_{2}^{'})((\dfrac{r}{2}+2(f^{2}+f^{'}))-(\dfrac{r}{2}+3(f^{2}+f^{'}))\eta(T)^{2})\nonumber\\&&+ k_{1}^{2}k_{2}\big(\dfrac{r}{2}+3(f^{2}+f^{'}))\eta(N)\eta(T)\bigg]B
\nonumber \\&-&
\bigg[ (k_{1}^{''}-2k_{1}^{3}-k_{1}k_{2}^{2})\big(\dfrac{r}{2}+3(f^{2}+f^{'}))\eta(N)\nonumber\\&&+ (2k_{1}^{'}k_{2}+k_{1}k_{2}^{'})(\dfrac{r}{2}+3(f^{2}+f^{'}))\eta(B) \bigg]\xi\nonumber\\&=&0\nonumber \\ \label{t3}
\end{eqnarray}
As a result, we can state the following theorem:

\begin{theorem} \label{anateo}
Let $\gamma(s):I\rightarrow M$ be a non-geodesic, arc-length parametrized Frenet curve in a $3$-dimensional $f$-Kenmotsu manifold $(M,\varphi,\xi,\eta,g)$. Then $\gamma(s)$ is a triharmonic Frenet curve if and only if the following equations hold,
 \begin{equation} 
  \begin{cases}
    k_{1}k_{1}^{'''}+2k_{1}^{'}k_{1}^{''}-2k_{1}^{3}k_{1}^{'}-k_{2}^{2}k_{1}k_{1}^{'}-k_{1}^{2}k_{2}k_{2}^{'}=0, \vspace{0.2cm} \\  
    \begin{cases}      
    k_{1}^{(4)}-10k_{1}^{2}k_{1}^{''}-6k_{2}^{2}k_{1}^{''}-4k_{1}k_{2}k_{2}^{''}-15k_{1}(k_{1}^{'})^{2}-12k_{2}k_{2}^{'}k_{1}^{'}-3k_{1}(k_{2}^{'})^{2}+k_{1}^{5}+k_{1}k_{2}^{4}\\+2k_{1}^{3}k_{2}^{2}+(k_{1}^{''}-2k_{1}^{3}-k_{1}k_{2}^{2})\big(\dfrac{r}{2}+2(f^{2}+f^{'})-\big(\dfrac{r}{2}+3(f^{2}+f^{'})\big)(\eta(T)^{2}+\eta(N)^{2})\big)\\-\big(k_{1}^{2}k_{2}\eta(T)+(2k_{2}k_{1}^{'}+k_{1}k_{2}^{'})\eta(N)\big)\big(\dfrac{r}{2}+3(f^{2}+f^{'})\big)\eta(B)=0, \end{cases}  \label{triharmonic} \vspace{0.2cm}\\
    \begin{cases}
    4k_{2}k_{1}^{'''}+6k_{1}^{''}k_{2}^{'}+4k_{1}^{'}k_{2}^{''}-9k_{1}^{2}k_{2}k_{1}^{'}-4k_{2}^{3}k_{1}^{'}-6k_{1}k_{2}^{2}k_{2}^{'}-k_{2}^{'}k_{1}^{3}+k_{1}k_{2}^{'''}\\+(2k_{1}^{'}k_{2}+k_{1}k_{2}^{'})\big(\dfrac{r}{2}+2(f^{2}+f^{'})-(\dfrac{r}{2}+3(f^{2}+f^{'})\big)(\eta(T)^{2}+\eta(B)^{2})\big)\\+\big(k_{1}^{2}k_{2}\eta(T)-(k_{1}^{''}-2k_{1}^{3}-k_{1}k_{2}^{2})\eta(B)\big)\big(\dfrac{r}{2}+3(f^{2}+f^{'})\big)\eta(N)=0 \end{cases} 
  \end{cases} 
\end{equation}
where $k_{1}=k_{1}(s)$ and $k_{2}=k_{2}(s)$ are the curvature and the torsion of $\gamma(s)$, respectively.
\end{theorem}

\begin{proof}
By taking the inner product of (\ref{t3}), respectively with $T,N,B$ vector fields, we get a system of differential equations given in (\ref{triharmonic}).
\end{proof}

Now, we investigate results of Theorem \ref{anateo} in four cases.\\

\textbf{Case I}:
If $k_{1}=constant>0$ then (\ref{triharmonic}) reduces to;
\begin{equation}
 \begin{cases}
      k_{1}^{2}k_{2}k_{2}^{'}=0, \vspace{0.3cm} \\
      \begin{cases}
       4k_{1}k_{2}k_{2}^{''}+3k_{1}(k_{2}^{'})^{2}-k_{1}^{5}-k_{1}k_{2}^{4}-2k_{1}^{3}k_{2}^{2}\\+(2k_{1}^{3}+k_{1}k_{2}^{2})(\dfrac{r}{2}+2(f^{2}+f^{'})-(\dfrac{r}{2}+3(f^{2}+f^{'}))(\eta(T)^{2}+\eta(N)^{2}))\\+(k_{1}^{2}k_{2}\eta(T)+k_{1}k_{2}^{'}\eta(N))(\dfrac{r}{2}+3(f^{2}+f^{'}))\eta(B)=0,  \end{cases}\vspace{0.3cm}\\
       \begin{cases}
       6k_{1}k_{2}^{2}k_{2}^{'}+k_{2}^{'}k_{1}^{3}-k_{1}k_{2}^{'''}\\-(k_{1}k_{2}^{'})\big(\dfrac{r}{2}+2(f^{2}+f^{'})-(\dfrac{r}{2}+3(f^{2}+f^{'}))(\eta(T)^{2}+\eta(B)^{2})\big)\\-(k_{1}^{2}k_{2}\eta(T)+(2k_{1}^{3}+k_{1}k_{2}^{2})\eta(B))(\dfrac{r}{2}+3(f^{2}+f^{'})\eta(N)=0.
       \end{cases} 
  \end{cases} \label{triharmoniccaseI}
\end{equation}
\noindent Hence, we give the following theorem;

\begin{theorem}
Let $\gamma(s)$ be a non-geodesic proper triharmonic curve with constant curvature $k_{1}$ and nonzero torsion $k_{2}$ in a $3$-dimensional $f$-Kenmotsu manifold.  Then $\gamma(s)$ is a Frenet helix. Further, $k_{1}$ and $k_{2}$ satisfy the  differential equation system:
\begin{equation}
 \begin{cases}
 \begin{cases}
(k_{1}^{2}+k_{2}^{2})^{2}=(2k_{1}^{2}+k_{2}^{2})\big(\dfrac{r}{2}+2(f^{2}+f^{'})-(\dfrac{r}{2}+3(f^{2}+f^{'}))(\eta(T)^{2}+\eta(N)^{2})\big)\\\hspace{2cm}+k_{1}k_{2}\eta(T)\eta(B)(\dfrac{r}{2}+3(f^{2}+f^{'})),\end{cases} \vspace{0.3cm} \\
\big(k_{1}k_{2}\eta(T)+(2k_{1}^{2}+k_{2}^{2})\eta(B)\big)(\dfrac{r}{2}+3(f^{2}+f^{'})\eta(N)=0. 
 \label{triharmoniccasI*}
\end{cases} 
\end{equation}
\end{theorem}

\begin{proof} Since $\gamma(s)$ is a proper triharmonic curve and from first equation of (\ref{triharmoniccaseI}); if curvature $k_{1}$ is a constant, then the torsion $k_{2}$ is necessarily constant, namely the curve is a Frenet helix. Thereafter, assuming that the
curvature $k_{1} $ and the torsion $k_{2}$ are constant, second and third equations of (\ref{triharmoniccaseI}) become, the first and the second equation of (\ref{triharmoniccasI*}), respectively. 

\end{proof}

\textbf{Case II}:
If $k_{1}=constant>0$ and $k_{2}=0$, then (\ref{triharmonic}) reduces to; 
\begin{equation}
 \begin{cases}       
        k_{1}^{5}-2k_{1}^{3}\big[\dfrac{r}{2}+2(f^{2}+f^{'})-(\dfrac{r}{2}+3(f^{2}+f^{'}))(\eta(T)^{2}+\eta(N)^{2})\big]=0, \vspace{0.3cm}\\
        2k_{1}^{3}(\dfrac{r}{2}+3(f^{2}+f^{'}))\eta(N)\eta(B)=0. 
  \end{cases} \label{triharmoniccaseII}
\end{equation}

\indent In the second equation of \ref{triharmoniccaseII}, since $k_{1}>0$; $(\dfrac{r}{2}+3(f^{2}+f^{'}))$ or $\eta(N)$ or $\eta(B)$ can be equal to zero. However if $\eta(N)=0,$ we obtain that $\gamma$ is a Legendre curve. But as we mentioned in Remark 2 for a Legendre curve in $f$-Kenmotsu manifold $\eta(N)=-1$. This is a contradiction with our assumption $\eta(N)=0.$ Therefore, this subcase was not taken into consideration. For this reason we examine Case II in two subcases.

\textbf{Subcase II-1}:
 If  $(\dfrac{r}{2}+3(f^{2}+f^{'}))=0$, then (\ref{triharmoniccaseII}) reduces to
 \begin{equation}  
 k_{1}^{2}+2(f^{2}+f^{'})=0. \nonumber 
   \end{equation}
Hence, we have the following theorem.
\begin{theorem} Let $\gamma(s)$ is a non-geodesic Frenet curve with constant curvature $k_{1}$ and zero torsion $k_{2}$ in a $3$-dimensional $f$-Kenmotsu manifold of constant scalar curvature $r=3k_{1}^{2}$. Then $\gamma(s)$ is a proper triharmonic curve for $f^{2}+f^{'}< 0$ if and only if $$f(s)=\frac{k_{1}}{\sqrt{2}}tan(\frac{1}{2}(\sqrt{2}k_{1}c_{1}-\sqrt{2}k_{1}s))$$  
where $c_{1} > s$ is an integration constant.

\end{theorem}

\textbf{Subcase II-2}:
 If $\eta(B)=0$, then (\ref{triharmoniccaseII}) reduces to 
 $$k_{1}^{5}-2k_{1}^{3}\big(\dfrac{r}{2}+2(f^{2}+f^{'})-(\dfrac{r}{2}+3(f^{2}+f^{'}))(\eta(T)^{2}+\eta(N)^{2})\big)=0.$$
 Then, we get;
 \begin{theorem} Let $\gamma(s)$ is a non-geodesic Frenet curve with constant curvature $k_{1}$ and zero torsion $k_{2}$ in a $3$-dimensional $f$-Kenmotsu manifold and $\eta(B)=0$. Then $\gamma(s)$ is a proper triharmonic curve if and only if $$k_{1}= \sqrt{ 2\big(\dfrac{r}{2}+2(f^{2}+f^{'})-(\dfrac{r}{2}+3(f^{2}+f^{'}))(\eta(T)^{2}+\eta(N)^{2})\big)}.$$  
 
 \end{theorem}

\textbf{Case III}:
If $k_{1}\neq  constant$ and $k_{2}=constant \neq 0$ , then (\ref{triharmonic}) reduces to;
\begin{equation}
  \begin{cases}
    k_{1}k_{1}^{'''}+2k_{1}^{'}k_{1}^{''}-2k_{1}^{3}k_{1}^{'}-k_{2}^{2}k_{1}k_{1}^{'}=0, \vspace{0.3cm} \\
    \begin{cases}
    k_{1}^{(4)}-10k_{1}^{2}k_{1}^{''}-6k_{2}^{2}k_{1}^{''}-15k_{1}(k_{1}^{'})^{2}+k_{1}^{5}+k_{1}k_{2}^{4}+2k_{1}^{3}k_{2}^{2}\\+(k_{1}^{''}-2k_{1}^{3}-k_{1}k_{2}^{2})\big(\dfrac{r}{2}+2(f^{2}+f^{'})-\big(\dfrac{r}{2}+3(f^{2}+f^{'})\big)(\eta(T)^{2}+\eta(N)^{2})\big)\\-\big(k_{1}^{2}k_{2}\eta(T)+2k_{2}k_{1}^{'}\eta(N)\big)\big(\dfrac{r}{2}+3(f^{2}+f^{'})\big)\eta(B)=0,\end{cases} \vspace{0.3cm}\\
     \begin{cases}
    4k_{2}k_{1}^{'''}-9k_{1}^{2}k_{2}k_{1}^{'}-4k_{2}^{3}k_{1}^{'}\\+(2k_{1}^{'}k_{2})\big(\dfrac{r}{2}+2(f^{2}+f^{'})+\big(\dfrac{r}{2}+3(f^{2}+f^{'})\big)(\eta(T)^{2}+\eta(B)^{2})\big)\\+\big(k_{1}^{2}k_{2}\eta(T)-(k_{1}^{''}-2k_{1}^{3}-k_{1}k_{2}^{2})\eta(B)\big)\big(\dfrac{r}{2}+3(f^{2}+f^{'})\big)\eta(N)=0.\end{cases}\label{case3}
  \end{cases} 
\end{equation}
So we have the following theorem.
\begin{theorem}
Let $\gamma(s)$ be a non-geodesic Frenet curve in a $3$-dimensional $f$-Kenmotsu manifold. Then $\gamma(s)$ is a proper triharmonic curve with nonconstant curvature $k_{1}$ and nonzero constant torsion $k_{2}$ if and only if the following equations are provided,
\begin{equation}
  \begin{cases}
    k_{1}k_{1}^{'''}+2k_{1}^{'}k_{1}^{''}-2k_{1}^{3}k_{1}^{'}-k_{2}^{2}k_{1}k_{1}^{'}=0, \vspace{0.3cm} \\
    \begin{cases}
    k_{1}^{(4)}-10k_{1}^{2}k_{1}^{''}-6k_{2}^{2}k_{1}^{''}-15k_{1}(k_{1}^{'})^{2}+k_{1}^{5}+k_{1}k_{2}^{4}+2k_{1}^{3}k_{2}^{2}\\+(k_{1}^{''}-2k_{1}^{3}-k_{1}k_{2}^{2})\big(\dfrac{r}{2}+2(f^{2}+f^{'})-\big(\dfrac{r}{2}+3(f^{2}+f^{'})\big)(\eta(T)^{2}+\eta(N)^{2})\big)\\-\big(k_{1}^{2}k_{2}\eta(T)+2k_{2}k_{1}^{'}\eta(N)\big)\big(\dfrac{r}{2}+3(f^{2}+f^{'})\big)\eta(B)=0,\end{cases} \vspace{0.3cm}\\
     \begin{cases}
    4k_{2}k_{1}^{'''}-9k_{1}^{2}k_{2}k_{1}^{'}-4k_{2}^{2}k_{1}^{'}\\+(2k_{1}^{'}k_{2})\big(\dfrac{r}{2}+2(f^{2}+f^{'})+\big(\dfrac{r}{2}+3(f^{2}+f^{'})\big)(\eta(T)^{2}+\eta(B)^{2})\big)\\+\big(k_{1}^{2}k_{2}\eta(T)-(k_{1}^{''}-2k_{1}^{3}-k_{1}k_{2}^{2})\eta(B)\big)\big(\dfrac{r}{2}+3(f^{2}+f^{'})\big)\eta(N)=0.\end{cases}
  \end{cases} 
\end{equation}

\end{theorem}

\textbf{Case IV}:
If $k_{1}\neq  constant$ and $k_{2}= 0$ , then (\ref{triharmonic}) reduces to;
\begin{equation} \label{case4}
  \begin{cases}
    k_{1}k_{1}^{'''}+2k_{1}^{'}k_{1}^{''}-2k_{1}^{3}k_{1}^{'}=0, \vspace{0.3cm} \\
     \begin{cases}
    k_{1}^{(4)}-10k_{1}^{2}k_{1}^{''}-15k_{1}(k_{1}^{'})^{2}+k_{1}^{5}\\+(k_{1}^{''}-2k_{1}^{3})\big(\dfrac{r}{2}+2(f^{2}+f^{'})-\big(\dfrac{r}{2}+3(f^{2}+f^{'})\big)(\eta(T)^{2}+\eta(N)^{2})\big)=0,  \end{cases} \vspace{0.3cm}\\
     
    (k_{1}^{''}-2k_{1}^{3})\big(\dfrac{r}{2}+3(f^{2}+f^{'})\big)\eta(B)\eta(N)=0  
  \end{cases} 
\end{equation}
\indent Since $(\dfrac{r}{2}+3(f^{2}+f^{'}))$ or $\eta(B)$ can be equal to zero in the third equation of (\ref{case4}) we examine Case IV in two subcases.\\

\textbf{Subcase IV-1}:
If $(\dfrac{r}{2}+3(f^{2}+f^{'}))=0,$ $k_{1}\neq  constant$ and $k_{2}= 0$ then (\ref{case4}) reduces to
\begin{equation} \label{23}
  \begin{cases}
    k_{1}k_{1}^{'''}+2k_{1}^{'}k_{1}^{''}-2k_{1}^{3}k_{1}^{'}=0, \vspace{0.3cm} \\
    k_{1}^{(4)}-10k_{1}^{2}k_{1}^{''}-15k_{1}(k_{1}^{'})^{2}+k_{1}^{5}+(k_{1}^{''}-2k_{1}^{3})\big(\dfrac{r}{2}+2(f^{2}+f^{'})\big)=0
  \end{cases} 
\end{equation}

Thus we obtain the following theorem.

\begin{theorem} \label{teo5}
Let $\gamma(s)$ be a non-geodesic Frenet curve with nonconstant curvature $k_{1}$ and zero torsion $k_{2}$ in a $3$-dimensional $f$-Kenmotsu manifold. Then $\gamma(s)$ is a proper triharmonic curve if and only if $$k_{1}(s)=\dfrac{\sqrt{5}}{s},$$ $$f(s)=\dfrac{9}{2s}$$ and $$r(s)=-\dfrac{189}{2s^2}.$$
\end{theorem}
\begin{proof}
We use the methods in reference \cite{Montaldo} to find the solutions system of differential equations (\ref{23}). Initially, the first equation of (\ref{23}) is multiplied by $5k_{1}$ and integrated, then we have
\begin{equation}
5k_{1}^{2}k_{1}^{''}-2k_{1}^{5}=c_{1} \label{31}
\end{equation} where $c_{1}$ a real constant. Similarly, when (\ref{31}) is multiplied by $2k_{1}^{-2}k_{1}^{'} $ and then integrated, we get
 \begin{equation}
5(k_{1}^{'})^{2}=k_{1}^{4}-2k_{1}^{-1}c_{1}+c_{2}, \label{32}
\end{equation}
for another real constants $c_{1}$ and $c_{2}$.
For an explicit solution, by considering $c_{1}=c_{2}=0$; $k_{1}(s)=\dfrac{\sqrt{5}}{s}$ is obtained as a solution of  (\ref{32}).\\
For $k_{1}^{''}\neq 2k_{1}^{3}$, by substituting $k_{1}(s)=\dfrac{\sqrt{5}}{s}$ and its derivatives into the second equation of (\ref{23}), we get $f(s)=\dfrac{9}{2s}$ and $r(s)=-\dfrac{189}{2s^2}.$

\end{proof}

\textbf{Subcase IV-2}:
If $\eta(B)=0,$ $k_{1}\neq  constant$ and $k_{2}= 0$ then (\ref{case4}) reduces to
\begin{equation} \label{26}
  \begin{cases}
    k_{1}k_{1}^{'''}+2k_{1}^{'}k_{1}^{''}-2k_{1}^{3}k_{1}^{'}=0, \vspace{0.3cm} \\
     \begin{cases}
    k_{1}^{(4)}-10k_{1}^{2}k_{1}^{''}-15k_{1}(k_{1}^{'})^{2}+k_{1}^{5}\\+(k_{1}^{''}-2k_{1}^{3})\big(\dfrac{r}{2}+2(f^{2}+f^{'})-\big(\dfrac{r}{2}+3(f^{2}+f^{'})\big)(\eta(T)^{2}+\eta(N)^{2})\big)=0.  \end{cases} 
  \end{cases} 
\end{equation}
Then, we conclude the following theorem.
\begin{theorem} \label{teo6}
Let $\gamma(s)$ be a non-geodesic Frenet curve with nonconstant curvature $k_{1}$, zero torsion $k_{2}$ and $\eta(B)=0$  in a $3$-dimensional $f$-Kenmotsu manifold. Then $\gamma(s)$ is a proper triharmonic curve if and only if following system of differential equations are satisfied;
\begin{equation} 
  \begin{cases}
    k_{1}k_{1}^{'''}+2k_{1}^{'}k_{1}^{''}-2k_{1}^{3}k_{1}^{'}=0, \vspace{0.3cm} \\
     \begin{cases}
    k_{1}^{(4)}-10k_{1}^{2}k_{1}^{''}-15k_{1}(k_{1}^{'})^{2}+k_{1}^{5}\\+(k_{1}^{''}-2k_{1}^{3})\big(\dfrac{r}{2}+2(f^{2}+f^{'})-\big(\dfrac{r}{2}+3(f^{2}+f^{'})\big)(\eta(T)^{2}+\eta(N)^{2})\big)=0.  \end{cases} 
  \end{cases} 
\end{equation}
\end{theorem}

\subsection{Triharmonic slant curves in $f$-Kenmotsu manifolds}

In the following subsection we derive the triharmonicity conditions for slant curves in $f$-Kenmotsu manifolds and discuss
the particular cases of  $k_{1}$ and $k_{2}.$

\begin{theorem} \label{slantanateo}
Let $\gamma(s)$ be a non-geodesic arc-length parametrized slant curve in a 3-dimensional $f$-Kenmotsu manifold $(M,\varphi,\xi,\eta,g).$ Then $\gamma(s)$ is a triharmonic curve if an only if
\begin{equation} \small
  \begin{cases}
    k_{1}k_{1}^{'''}+2k_{1}^{''}k_{1}^{'}-2k_{1}^{3}k_{1}^{'}-k_{2}^{2}k_{1}k_{1}^{'}-k_{1}^{2}k_{2}k_{2}^{'}=0, \vspace{0.3cm} \\
    \begin{cases}
    k_{1}^{(4)}-10k_{1}^{2}k_{1}^{''}-6k_{2}^{2}k_{1}^{''}-4k_{1}k_{2}k_{2}^{''}-15k_{1}(k_{1}^{'})^{2}-12k_{2}k_{2}^{'}k_{1}^{'}-3k_{1}(k_{2}^{'})^{2}\\+k_{1}^{5}+k_{1}k_{2}^{4}+2k_{1}^{3}k_{2}^{2}+(k_{1}^{''}-2k_{1}^{3}-k_{1}k_{2}^{2})\big(\dfrac{r}{2}+2(f^{2}+f^{'})\\-\big(\dfrac{r}{2}+3(f^{2}+f^{'})\big)((cos\theta)^{2}+\dfrac{f^{2}}{k_{1}^{2}}(sin\theta)^{4})\big)\\-\big(k_{1}^{2}k_{2}cos\theta-(2k_{2}k_{1}^{'}+k_{1}k_{2}^{'})\dfrac{f}{k_{1}}(sin\theta)^{2}\big)\big(\dfrac{r}{2}+3(f^{2}+f^{'})\big)(\dfrac{\mid sin\theta \mid}{k_{1}}\sqrt{k_{1}^{2}-f^{2}(sin\theta)^{2}})=0, 
    \end{cases}    \vspace{0.3cm}\\
    \begin{cases}
  4k_{2}k_{1}^{'''}+6k_{1}^{''}k_{2}^{'}+4k_{1}^{'}k_{2}^{''}-9k_{1}^{2}k_{2}k_{1}^{'}-4k_{2}^{3}k_{1}^{'}-6k_{1}k_{2}^{2}k_{2}^{'}-k_{2}^{'}k_{1}^{3}+k_{1}k_{2}^{'''}\\+(2k_{1}^{'}k_{2}+k_{1}k_{2}^{'})\big(\dfrac{r}{2}+2(f^{2}+f^{'})\\+\big(\dfrac{r}{2}+3(f^{2}+f^{'})\big)((cos\theta)^{2}+\dfrac{(sin\theta)^{2}}{k_{1}^{2}}({k_{1}^{2}-f^{2}(sin\theta)^{2}}))\big)\\-\big(k_{1}^{2}k_{2}cos\theta-(k_{1}^{''}-2k_{1}^{3}-k_{1}k_{2}^{2})\dfrac{\mid sin\theta \mid }{k_{1}}\sqrt{k_{1}^{2}-f^{2}(sin\theta)^{2}}\big)\big(\dfrac{r}{2}+3(f^{2}+f^{'})\big)\dfrac{f}{k_{1}}(sin\theta)^{2}=0, \end{cases} \label{slanttriharmonic} 
  \end{cases} 
\end{equation}
where $k_{1}=k_{1}(s)$ and $k_{2}=k_{2}(s)$ are the curvature and the torsion of $\gamma(s)$, respectively.
\end{theorem}
Here we examine some cases for the triharmonic slant curves in $3$-dimensional $f$-Kenmotsu manifold.\\

\textbf{Case I}:
If $k_{1}=constant>0$ and $k_{2}\neq 0$, then (\ref{slanttriharmonic}) reduces to;
\begin{equation} \small
 \begin{cases}
      k_{1}^{2}k_{2}k_{2}^{'}=0, \vspace{0.3cm} \\
       \begin{cases}
           -4k_{1}k_{2}k_{2}^{''}-3k_{1}(k_{2}^{'})^{2}+k_{1}^{5}+k_{1}k_{2}^{4}+2k_{1}^{3}k_{2}^{2}-(2k_{1}^{3}+k_{1}k_{2}^{2})\big(\dfrac{r}{2}+2(f^{2}+f^{'})\big)\\+(2k_{1}^{3}+k_{1}k_{2}^{2})\big(\dfrac{r}{2}+3(f^{2}+f^{'})\big)((cos\theta)^{2}+\dfrac{f^{2}}{k_{1}^{2}}(sin\theta)^{4})\\-\big(k_{1}^{2}k_{2}cos\theta-(k_{1}k_{2}^{'})\dfrac{f}{k_{1}}(sin\theta)^{2}\big)\big(\dfrac{r}{2}+3(f^{2}+f^{'})\big)(\dfrac{\mid sin\theta \mid }{k_{1}}\sqrt{k_{1}^{2}-f^{2}(sin\theta)^{2}})=0, 
          \end{cases}    \vspace{0.3cm}\\
           \begin{cases}
           -6k_{1}k_{2}^{2}k_{2}^{'}-k_{2}^{'}k_{1}^{3}+k_{1}k_{2}^{'''}+(k_{1}k_{2}^{'})\big(\dfrac{r}{2}+2(f^{2}+f^{'})\big)\\+(k_{1}k_{2}^{'})\big(\dfrac{r}{2}+3(f^{2}+f^{'})\big)((cos\theta)^{2}+(\dfrac{(sin\theta)^{2}}{k_{1}^{2}}({k_{1}^{2}-f^{2}(sin\theta)^{2}}))\\-\big(k_{1}^{2}k_{2}cos\theta+(2k_{1}^{3}+k_{1}k_{2}^{2})\dfrac{\mid sin\theta \mid }{k_{1}}\sqrt{k_{1}^{2}-f^{2}(sin\theta)^{2}}\big)\big(\dfrac{r}{2}+3(f^{2}+f^{'})\big)\dfrac{f}{k_{1}}(sin\theta)^{2}=0. \end{cases} 
  \end{cases} \label{slanttriharmoniccaseI}
\end{equation}
\noindent Then from the first equation of (\ref{slanttriharmoniccaseI}), we obtain that $k_{2}(s)=constant$ so we have the following theorem from Case I;

\begin{theorem}
Let $\gamma(s)$ be a non-geodesic proper triharmonic slant curve in a $3$-dimensional $f$-Kenmotsu manifold with constant curvature $k_{1}$ and nonzero torsion $k_{2}$.  Then $\gamma(s)$ is a slant helix. Besides, $k_{1}$ and $k_{2}$ satisfy the following system of differential equations;
\begin{equation}
 \begin{cases}
  \begin{cases}        
            (k_{1}^{2}+k_{2}^{2})^{2}=(2k_{1}^{2}+k_{2}^{2})(\dfrac{r}{2}+2(f^{2}+f^{'}))-\bigg[(2k_{1}^{2}+k_{2}^{2})((cos\theta)^{2}+\dfrac{f^{2}}{k_{1}^{2}}(sin\theta)^{4})\\\hspace{2.2cm}-k_{1}k_{2}cos\theta \dfrac{\mid sin\theta \mid}{k_{1}}\sqrt{k_{1}^{2}-f^{2}(sin\theta)^{2}}\bigg](\dfrac{r}{2}+3(f^{2}+f^{'})) \end{cases} \vspace{0.3cm}\\          
           \big(k_{1}k_{2}cos\theta+(2k_{1}^{2}+k_{2}^{2})\dfrac{\mid sin\theta \mid}{k_{1}}\sqrt{k_{1}^{2}-f^{2}(sin\theta)^{2}}\big)\big(\dfrac{r}{2}+3(f^{2}+f^{'})\big)\dfrac{f}{k_{1}}(sin\theta)^{2}=0.
 \end{cases} 
\end{equation}
\end{theorem}

\textbf{Case II}:
If $k_{1}=cons.>0$ and $k_{2}=0$, then (\ref{slanttriharmonic}) reduces to; 
\begin{equation}
 \begin{cases}
           k_{1}^{5}-2k_{1}^{3}\big(\dfrac{r}{2}+2(f^{2}+f^{'})\big)+2k_{1}^{3}\big(\dfrac{r}{2}+3(f^{2}+f^{'})\big)((cos\theta)^{2}+\dfrac{f^{2}}{k_{1}^{2}}(sin\theta)^{4})=0, 
               \vspace{0.3cm}\\
           2k_{1}^{3}\dfrac{\mid sin\theta \mid }{k_{1}}\sqrt{k_{1}^{2}-f^{2}(sin\theta)^{2}}\big(\dfrac{r}{2}+3(f^{2}+f^{'})\big)\dfrac{f}{k_{1}}(sin\theta)^{2}=0. \label{slanttriharmoniccaseII}
 \end{cases} 
\end{equation}
Then, we obtain the following theorem.
\begin{theorem}
Let $(M,\varphi,\xi,\eta,g)$ be a 3-dimensional $f$-Kenmotsu manifold and $\gamma:I\rightarrow M$ be a non-geodesic slant curve. Then for $k_{1}=cons.>0$ and $k_{2}=0,$ $\gamma$ is a proper triharmonic curve if and only if $M$ is of constant scalar curvature $r= -3k_{1}^{2}$ and 
 $$f(s)=\frac{k_{1}}{\sqrt{2}}tan(\frac{1}{2}(\sqrt{2}k_{1}c_{1}-\sqrt{2}k_{1}s)). $$
\end{theorem}

\subsection{Triharmonic Legendre curves in $f$-Kenmotsu manifolds}
In this subsection we discuss the triharmonicity condition for Legendre curves in $f$-Kenmotsu manifolds.
\begin{theorem} \label{Legteo}
There is no triharmonic Legendre curve in a 3-dimensional $f$-Kenmotsu manifold $(M,\varphi,\xi,\eta,g).$ 
\end{theorem}

\begin{proof}
When Theorem \ref{anateo} is rewritten according to the properties of the Legendre curve given in Remark \ref{legremark}, following differential equations are obtained;
\begin{equation}
  \begin{cases}
    k_{1}k_{1}^{'''}+2k_{1}^{''}k_{1}^{'}-2k_{1}^{3}k_{1}^{'}=0, \vspace{0.3cm} \\
    k_{1}^{(4)}-10k_{1}^{2}k_{1}^{''}-15k_{1}(k_{1}^{'})^{2}+k_{1}^{5}-(k_{1}^{''}-2k_{1}^{3})(f^{2}+f^{'})=0. \label{legendretriharmonic}
  \end{cases} 
\end{equation}
From the first equation of (\ref{legendretriharmonic}), via differential equation solving methods (such as in Theorem \ref{teo5}), $k_{1}(s)=\dfrac{\sqrt{5}}{s}$ obtained, \cite{Montaldo}. By using $k_{1}(s)=\dfrac{\sqrt{5}}{s}$ and its derivatives in the second equation of (\ref{legendretriharmonic}), we have $f(s)=\dfrac{9}{2s},$ but this is a contradiction with the definition of Legendre curves in $3$-dimensional $f$-Kenmotsu manifold.
\end{proof}

\section{Conclusion}

There are limited number of studies investigating triharmonic curves as opposed to harmonic and biharmonic curves. In this study, triharmonic curves are studied in three dimensional $f$-Kenmotsu manifolds and original theorems related to slant and Legendre curves are obtained. For this reason, we think that our study will be a guide not only for the author who studies $k$-harmonic curves, but also for the authors who study slant and Legendre curves.


\begin{thebibliography}{99}

\bibitem{blair} Blair D. E.: Contact manifolds in Riemannian geometry. Springer, Berlin, Heidelberg. (1976).

\bibitem{munteanu} Călin C., Crasmareanu M., Munteanu M.I.: Slant curves in three-dimensional f-Kenmotsu manifolds. J. Math. Anal. Appl. \textbf{394}, 400-407 (2012).

\bibitem{EellsSamp} Eells J., Sampson J.H.: Harmonic mappings of Riemannian manifolds. American Journal of Mathematics. \textbf{86}, 109–160 (1964).

\bibitem{EellsSamp2} Eells J., Sampson J. H.: Energie et Deformations en Geometrie Differentielle. Ann. Inst. Fourier \textbf{14}, 61–69 (1964).

\bibitem{EllsLamaire} Eells J., Lemaire L.: Selected topics in harmonic maps. American Mathematical Soc. \textbf{50}, (1983). 

\bibitem{Jiang} Jiang G.Y.: $2$-Harmonic maps and their first and second variational formulas. Chinese Ann Math Ser A  \textbf{7}, 389-402 (1986).

\bibitem{kenmo} Kenmotsu K.: A class of almost contact Riemannian manifolds, Tohoku Mathematical Journal, Second Series  \textbf{24(1)}, 93-103 (1972).

\bibitem{maeta} Maeta S.: $k$-Harmonic maps into a Riemannian Manifold with constant sectional curvature. Proc. Amer. Math. Soc. \textbf{140}, 1835–1847 (2012).

\bibitem{mangione} Mangione V.: Harmonic maps and stability on-Kenmotsu Manifolds. Int. Journal of Mathematics and Mathematical Sciences. 2008.

\bibitem{Montaldo} Montaldo S., Pampano A.: Triharmonic curves in 3-dimensional homogeneous spaces. Mediterranean Journal of Mathematics. \textbf{18(5)}, 198 (2021). https://doi.org/10.1007/s00009-021-01837-y


\bibitem{Perk} Perktaş S. Y., Acet B. E.: Ouakkas S., On Biharmonic and Biminimal Curves in 3-dimensional $f$-Kenmotsu Manifolds. Fundamentals of Contemporary Mathematical Sciences.  \textbf{1(1)}, 14-22 (2020).



\end{thebibliography}

\end{document}